\documentclass[11pt,reqno]{amsart}

\usepackage{amssymb}
\usepackage{amscd}
\usepackage{amsfonts}
\usepackage{mathrsfs}
\usepackage{setspace}
\usepackage{version}
\usepackage{mathtools}
\usepackage{enumerate}
\usepackage[english]{babel}
\usepackage{xcolor}
\usepackage{graphicx}

\newtheorem{theorem}{Theorem}[section]
\newtheorem{lemma}[theorem]{Lemma}
\newtheorem{proposition}[theorem]{Proposition}

\theoremstyle{definition}
\newtheorem{definition}[theorem]{Definition}
\newtheorem{assumption}[theorem]{Assumption}
\newtheorem{example}[theorem]{Example}
\newtheorem{examples}[theorem]{Examples}

\theoremstyle{remark}
\newtheorem{remark}[theorem]{Remark}
\numberwithin{equation}{section}

\newcommand{\field}[1]{\mathbb{#1}}

\newcommand{\R}{\field{R}}
\newcommand{\N}{\field{N}}

\begin{document}

	\title[]{On the representation of weakly maxitive monetary risk measures and their rate functions}
	
%
	
		\author{Jos\'e M.~Zapata}
	\address{Centro Universitario de la Defensa. Universidad Politécnica de Cartagena. 
c/ Coronel López Peña S/N, Santiago de La Ribera, 30720, Murcia, Spain}
	\email{jose.zapata@cud.upct.es}

		\date{\today}

	\thanks{I would like to thank Michael Kupper and Henri Comman for valuable comments, and an anonymous reviewer for his suggestions, which have helped to improve the presentation of the article.}
	
	\subjclass[2010]{}

	\begin{abstract} 
	The present paper provides a representation result for monetary risk measures (i.e., monotone translation invariant functionals) satisfying a weak maxitivity property.   
	This result can be understood as a functional analytic generalization of G\"{a}rtner-Ellis  large deviations theorem.   
	In contrast to the classical G\"{a}rtner-Ellis theorem, the rate function is computed on an arbitrary set of continuous real-valued functions rather than the dual space.  
  		As an application of the main result, we establish a large deviation result for sequences of sublinear expectations on regular Hausdorff topological spaces. 

		
		
	\end{abstract}

	\maketitle
	
	\setcounter{tocdepth}{1}

\section{Introduction}
 The theory of large deviations studies the asymptotic tail behaviour of sequences of random variables.  
 The earliest developments of this theory arose in the context of ruin theory in actuarial science \cite{cramer,cramer2}, while Varadhan and Donsker \cite{donsker1,varadhan1966asymptotic} systematically developed the modern framework of this field.   
 Next, we recall the basic principles of large deviations theory; we refer to the excellent monograph~\cite{dembo} for further details and historical background. 
 Let $E$ be a regular Hausdorff topological space and a sequence $(X_n)_{n\in\N}$ of $E$-valued random variables defined in a common probability space $(\Omega,\mathcal{F},\mathbb{P})$. 
 The sequence $(X_n)_{n\in\N}$ is said to satisfy the \emph{large deviation principle} (LDP) with rate function $I\colon E\to[0,\infty]$ if
 \[
-\underset{x\in {\rm int}(A)}\inf I(x)\le \liminf_{n\to\infty}\frac{1}{n}\log\mathbb{P}(X_n\in A)\\ 
\le \limsup_{n\to\infty}\frac{1}{n}\log\mathbb{P}(X_n\in A) \le -\underset{x\in {\rm cl}(A)}\inf I(x)
 \]
 for all Borel set $A\subset E$.\footnote{We denote by ${\rm int}(A)$ and ${\rm cl}(A)$ the topological interior and closure of $A\subset E$, respectively.} 
 The Varadhan's integral lemma asserts that a sequence $(X_n)_{n\in\N}$ that satisfies the LDP with rate function $I(\cdot)$ also satisfies the \emph{Laplace principle} (LP) with rate function $I(\cdot)$, that is, 
 \[
 \psi(f)=\underset{x\in E}\sup\{f(x)-I(x)\}
 \]
 for all $f\in C_b(E)$.\footnote{We denote by $C_b(E)$ the set of all bounded continuous real-valued functions on $E$.} 
 Here, we denote by $\psi(f)=\lim_{n\to\infty}\tfrac{1}{n}\log\mathbb{E}_{\mathbb{P}}[e^{n f(X_n)}]$ the asymptotic entropy of $f$.\footnote{The limit in $\psi(f)$ exists for all $f\in C_b(E)$ if $(X_n)$ satisfies the LDP.}   
The converse of the Varadhan's integral lemma also holds true under additional regularity conditions; Bryc \cite{bryc} established this statement assuming that $I(\cdot)$ has compact sublevel sets,  Comman~\cite{comman} proved the same assuming that $E$ is normal. 
 In addition, Bryc's theorem states that both the LDP and LP hold with the rate function $I(x)=\sup_{f\in C_b(E)}\{f(x)-\psi(f)\}$ if the sequence $(X_n)_{n\in\N}$ is exponentially tight\footnote{I.e., for all positive number $M$ there exists $K\subset E$ compact such that $\limsup_{n\to\infty}\frac{1}{n}\log\mathbb{P}(X_n\in K^c) \le -M$.}.   
 The functional $\psi\colon C_b(E)\to\R$ has some properties which are crucial in this theory. 
 First, $\psi$ is a monetary risk measure, that is, it is monotone (i.e., $\psi(f)\le\psi(g)$ whenever $f\le g$) and translation invariant (i.e., $\psi(f+c)=\psi(f)+c$).  
 Second, the asymptotic entropy $\psi$ has the remarkable property of being maxitive (i.e., $\psi(f\vee g)\le \psi(f)\vee\psi(g)$). 
 Actually, the properties of a maxitive monetary risk measure are sufficient to prove generalized versions of all the basic results listed above covering this type of functionals, striping away any probabilistic aspect of the theory.  
 In fact, Bell and Bryc~\cite{bell} introduced and studied a general LP for monetary risk measures\footnote{Bell and Bryc~\cite{bell} uses the term Varadhan functional rather than monetary risk measure. Here, we use the term monetary risk measure to build a bridge towards risk analysis.} on $C_b(E)$ and, more recently, Kupper and Zapata \cite{kupper} have formulated also a general LDP for this kind of functionals and extended to this general setting the Varadhan-Bryc equivalence between LDP and LP,  and the Bryc's theorem.    
  
  As a continuation of the research in \cite{kupper}, the present paper aims to identify new situations where a monetary risk measure satisfies the LP and the LDP and, in particular, how to compute the rate function $I(\cdot)$. 
  While our main result applies to general monetary
risk measures satisfying a weak form of maxitivity, we explain it now for easier readability for the case of the asymptotic entropy $\psi(f)=\lim_{n\to\infty}\tfrac{1}{n}\log\mathbb{E}_{\mathbb{P}}[e^{n f(X_n)}]$, where for simplicity in the exposition we assume that the latter limit exists for all real-valued continuous function $f$.  
 In the following we fix an arbitrarily given non-empty set $\mathcal{H}$ of continuous real-valued functionals on $E$, and consider the corresponding conjugate $\psi_{\mathcal{H}}^\ast$ which is defined by $\psi_{\mathcal{H}}^\ast(x)=\sup_{f\in\mathcal{H}}\{f(x)-\psi(f)\}$. 
 We want to establish sufficient conditions so that $(X_n)_{n\in\N}$ satisfies the LDP with rate function $\psi_{\mathcal{H}}^\ast$.  
 We say that a point $x\in E$ is $\mathcal{H}$-exposed for $\psi_{\mathcal{H}}^\ast$ if there exists a function $f\in\mathcal{H}$ such that
\begin{equation}\label{eq:exposingConf}
 \psi^\ast_{\mathcal{H}}(y) - f(y) >   \psi^\ast_{\mathcal{H}}(x)-f(x) \quad\mbox{ for all }y\neq x.
\end{equation}  
The interpretation of the exposing condition~\eqref{eq:exposingConf} is that, for certain constant $c$, the curve $y \mapsto f(y) + c$ lies strictly below the curve $\psi_{\mathcal{H}}^\ast$ on  $E\setminus\{x\}$, and agrees with it at $x$.  
Denote by $\mathscr{E}$ the set of all $\mathcal{H}$-exposed points of $E$. 
In the special case of the asymptotic entropy, the main result of this paper reads as follows. 
\begin{theorem}\label{thm:linear}
Suppose that the sequence $(X_n)_{n\in\N}$ is exponentially tight.  
Then: 
\begin{itemize}
\item[(i)] For every closed set $C\subset E$, we have the upper bound 
\[
\limsup_{n\to\infty}\frac{1}{n}\log\mathbb{P}(X_n\in C)\le -\underset{y\in C}\inf \psi_{\mathcal{H}}^\ast(x).
\]
\item[(ii)] For every open set $O\subset E$, we have the lower bound
\[
 -\underset{y\in O\cap\mathscr{E}}\inf \psi_{\mathcal{H}}^\ast(x)\le \liminf_{n\to\infty}\frac{1}{n}\log\mathbb{P}(X_n\in O).
\]
\item[(iii)]
If, moreover, 
\begin{equation}\label{eq:richExpPoints}
\underset{x\in O}\inf \psi_{\mathcal{H}}^\ast(x)=\underset{x\in O\cap{\mathscr{E}}}\inf \psi_{\mathcal{H}}^\ast(x)\quad\mbox{for all }O\subset E\mbox{ open,}
\end{equation}
then $(X_n)_{n\in\N}$ satisfies the LDP and LP with rate function $\psi_{\mathcal{H}}^\ast$. 
\end{itemize}
\end{theorem} 
Of great importance,   
 G\"{a}rtner-Ellis theorem provides the LDP for sequences  of random variables with values on a topological vector space by testing the rate function on the dual space.  
This well-known result turns out to be a particular instance of Theorem \ref{thm:linear} above for the special case when $E$ is a topological vector space and the particular choice $\mathcal{H}=E^\ast$; cf.~\cite[Theorem 4.5.20]{dembo}.  
One of the novelties of Theorem \ref{thm:linear} is that it allows for different options for the testing set $\mathcal{H}$, instead of limiting ourselves to the dual space.    
Moreover, it is well-known that G\"{a}rtner-Ellis theorem does not cover all cases in which a LDP exists; there are examples for which the LDP holds, but it does not follow from this basic result; see~\cite[Remarks(d),~p.~45]{dembo} and \cite{chen0}.  
In contrast, Theorem \ref{thm:linear} above  allows for arbitrary choices of $\mathcal{H}$, covering  situations where G\"{a}rtner-Ellis theorem fails. 
This is illustrated in Example \ref{ex:classical} below, where we have a situation where a LDP is not covered by  G\"{a}rtner-Ellis theorem, but it is captured by choosing a family $\mathcal{H}$ of inverted v-shaped functions.     

Whereas we have stated above our main result for the particular functional $\psi(f)=\lim_{n\to\infty}\tfrac{1}{n}\log\mathbb{E}_{\mathbb{P}}[e^{n f(X_n)}]$, it applies to very general functionals  allowing to cover some non-standard setups as those  in~\cite{backhoff,chen,eckstein,follmerII,lacker,tan,yan}.    
 For instance, in situations with model uncertainty, one may be interested in considering a set $\mathcal{P}$ of probability measures rather than a single probability measure $\mathbb{P}$; see \cite{chen,tan} and references therein.  
 This situation is covered by our main result by considering the robust asymptotic entropy  $\psi_{\mathcal{P}}(f)=\limsup_{n\to\infty}\tfrac{1}{n}\log\sup_{Q\in \mathcal{P}}\mathbb{E}_Q[e^{n f(X_n)}]$. In particular, we extend to infinite dimensional spaces the version of  G\"{a}rtner-Ellis theorem for sequences of sublinear expectations on $\R^d$ proven in~\cite{tan} with the advantage that the rate function is now tested on  arbitrary sets of continuous functions.   
 
 The present approach fully relies on topological and order properties and, in particular, it is not needed an underlying probability space. 
 We emphasize that the existing proofs of G\"{a}rtner-Ellis theorem are based on probability concepts such as the Radon-Nykodym derivative which are not needed here; cf.~\cite[Theorem 4.5.20]{dembo}.  
 In line with \cite{kupper2,puhalskii}, the machinery is taken from maxitive integration. 
 More specifically, as in \cite{kupper2} we use the convex integral introduced by Cattaneo~\cite{cattaneo}, which is conceptually related to the idempotent integral in tropical mathematics~\cite{maslov} and can be obtained as a transformation of the Shilkret integral~\cite{shilkret}. 
 In particular, we rely on the duality bounds for convex integrals  and the  convex integral representation of weakly maxitive monetary risk measures provided in \cite{kupper2}.   

The paper is organized as follows. 
In Section \ref{sec:preliminaries} we give some preliminaries on maxitive integration. 
In Section \ref{sec:representation} we focus on the integral representation of weakly maxitive monetary risk measures.  
In Section \ref{sec:main} we state and prove the main result of this paper. 
Finally, in Section \ref{sec:LDsublinearExp} we apply the main result to obtain a large deviations result  for sequences of sublinear expectations on regular Hausdorff topological spaces.

\section{Preliminaries on maxitive integration}\label{sec:preliminaries}
Throughout this paper let $E$ be a regular Hausdorff  topological space with Borel $\sigma$-algebra $\mathcal{B}(E)$.   
We always make the convention that $-\infty \cdot 0=0$. 
Then, given a function $f\colon E \to\R\cup \{-\infty\}$, the function $f 1_A - \infty 1_{A^c}$ takes the same values as $f$ on $A\subset E$ and the value $-\infty$ on $A^c$.

A set function $J\colon \mathcal{B}(E)\to [-\infty,0]$ is said to be a \emph{concentration} if:
\begin{itemize}
\item[(a1)] $J_\emptyset=-\infty$, $J_E=0$,
\item[(a2)] $J_A\le J_B$ whenever $A\subset B$.
\end{itemize}
We say that $J$ is \emph{maxitive} if:
\begin{itemize}
\item[(a3)] $J_{A\cup B}\le J_A\vee J_B$.
\end{itemize}
Denote by $B(E)$ the set of all Borel measurable functions $f\colon E\to \mathbb{R}\cup\{-\infty\}$. 
We define the \emph{convex integral} of $f\in B(E)$ with respect to the concentration $J$ as
\footnote{It is not difficult to show that $\phi_J(f)=\sup_{c\in\mathbb{R}}\left\{c + J_{\{f> c\}}\right\}=\sup_{c\in\mathbb{R}}\left\{c + J_{\{f\ge  c\}}\right\}$, i.e.~we obtain an equivalent definition of $\phi_J(f)$ if the strict inequality in \eqref{eq:maxitiveInt} is replaced by a non-strict inequality.}
\begin{equation}\label{eq:maxitiveInt}
\phi_J(f)=\underset{c\in\mathbb{R}}\sup\left\{c + J_{\{f> c\}}\right\}.
\end{equation}  
The following properties (b1)--(b4) below are provided in \cite{cattaneo} and their proofs easily adapt to the present setting.  
 The proof of (b5) can be found in~\cite[Lemma 4.1]{kupper2}.    
\begin{itemize}
\item[(b1)] $\phi_J(-\infty 1_{A^c})=J_A$ for all $A\in\mathcal{B}(E)$, 
\item[(b2)] $\phi_J(0)=0$,
\item[(b3)] $\phi_J(f+c)=\phi_J(f)+c$ for all constant $c\in \mathbb{R}$,
\item[(b4)] $\phi_J(f)\le \phi_J(g)$ whenever $f\le g$,
\item[(b5)] $\underset{n\to\infty}\lim \phi_J(f \wedge n)=\phi_J(f)$ and  $\underset{n\to\infty}\lim \phi_J(f \vee -n)=\phi_J(f)$.
\end{itemize}
If $J$ is maxitive, then the following hold:
\begin{itemize}
\item[(b6)] $\phi_J$ is maxitive, i.e. $\phi_J(f\vee g)\le \phi_J(f)\vee \phi_J(g)$, 
\item[(b7)] $\phi_J$ is convex, i.e. $\phi_J(\lambda f + (1-\lambda) g)\le \lambda \phi_J(f) + (1-\lambda) \phi_J(g)$ for all $0\le \lambda \le 1$.
\end{itemize}
The proofs of (b6) and (b7) can be found in \cite[Corollary 5]{cattaneo} and \cite[Theorem 7]{cattaneo}, respectively. 
Denote by ${L}(E)$ the set of all lower semicontinuous functions $f\colon E\to \mathbb{R}\cup\{-\infty\}$, and by ${U}(E)$ the set of all upper semicontinuous functions $f\colon E\to \mathbb{R}\cup\{-\infty\}$. 
The following duality bounds were proved in \cite[Theorem 3.4]{kupper2}.\footnote{\cite[Theorem 3.4]{kupper2} deals with functions which are increasing with respect to a given preorder. 
To apply \cite[Theorem 3.4]{kupper2} here, we consider the trivial preoreder, i.e. $x\le y$ whenever $x=y$.} 
\begin{theorem}\label{thm:dualityBounds}
Let $J$ be a concentration, and $I\colon E\to [0,\infty]$ a function. 
Then, the following equivalences hold. 
First, 
\begin{equation}\label{eq:LowerLDP}
-\underset{x\in O}\inf I(x)\le J_O\quad\mbox{ for all open set }O\subset E
\end{equation}
if and only if
\begin{equation}\label{eq:LowerLP}
\phi_J(f)\ge \underset{x\in E}\sup\{f(x)-I(x)\}\quad\mbox{ for all }f\in  {L}(E). 
\end{equation}
Second, 
\begin{equation}\label{eq:UpperLDP}
J_C\le -\underset{x\in C}\inf I(x)\quad\mbox{ for all closed set }C\subset E
\end{equation}
if and only if
\begin{equation}\label{eq:UpperLP}
\phi_J(f)\le \underset{x\in E}\sup\{f(x)-I(x)\}\quad\mbox{ for all }f\in  {U}(E). 
\end{equation}
\end{theorem}

The minimal rate function $I_{\min}\colon E\to [0,\infty]$ associated with a concentration $J$ is defined as
\begin{equation}\label{eq:minRate}
I_{\min}(x):=\underset{f\in {L}(E)}\sup\{f(x)-\phi_J(f)\}.
\end{equation}

\begin{proposition}\label{prop:lowerBounds}
Let $J$ be a concentration and suppose that $I_{\min}$ is defined as in \eqref{eq:minRate}. 
Then, $I_{\min}$ is the smallest mapping $I\colon E\to[0,\infty]$ which satisfies the equivalent inequalities \eqref{eq:LowerLDP} and \eqref{eq:LowerLP}. 
\end{proposition}

The minimal rate function has the following representation; see~\cite[Lemma 3.5]{kupper2}.  
\begin{proposition}\label{prop:representationRate}
Let $J$ be a concentration and suppose that $I_{\min}$ is defined as in~\eqref{eq:minRate}. 
For all $x\in E$ it holds
\[
-I_{\min}(x)=\underset{U\in\mathcal{U}_x}\inf J_U,
\]
where $\mathcal{U}_x$ is a base of open neighborhoods of $x\in E$. 
\end{proposition}
The following notion was introduced in \cite{kupper2}. 
\begin{definition}
A concentration $J$ is said to be \emph{weakly maxitive} if 
\[
J_C\le \vee_{i=1}^N J_{O_i} \mbox{ 
for all $C\subset E$ closed, $O_1,O_2,\ldots,O_N\subset E$ open, $N\in\N$, such that $C\subset\cup_{i=1}^N O_i$.} 
\]
\end{definition}
Notice that every maxitive concentration $J$ is also weakly maxitive.

\begin{definition}\label{def:tightness}
We say that a concentration $J$ is \emph{tight} is for every $n\in\mathbb{N}$ there exists a compact set $K\subset E$ such that $J_{K^c}<-n$.
\end{definition}

The following result was provided in \cite[Theorem 4.1]{kupper} (see~\cite[Remark 4.2]{kupper}) and \cite[Corollary 3.10]{kupper2} under slightly different assumptions. 
For the sake of completeness, we provide a short proof.  
\begin{proposition}\label{prop:Bryc}
Let $J$ be a concentration and suppose that $I_{\min}$ is defined as in \eqref{eq:minRate}. 
If $J$ is tight and weakly maxitive, then $I_{\min}$ satisfies \eqref{eq:LowerLDP}, \eqref{eq:LowerLP}, \eqref{eq:UpperLDP}, and \eqref{eq:UpperLP}. 
\end{proposition}
\begin{proof}
In view of Theorem \ref{thm:dualityBounds} and Proposition \ref{prop:lowerBounds}, it is enough to show \eqref{eq:UpperLDP}.  
Suppose that $C\subset E$ is closed. 
Fix $\varepsilon>0$. 
Since $J$ is tight, there exists a compact set such that $-J_{K^c}\ge \varepsilon^{-1}$.   
Due to Proposition~\ref{prop:representationRate} and by compactness, there exists $x_1,\ldots,x_N\in K\cap C$ and open sets $U_1,\ldots,U_N$ such that $x_i\in U_i\subset E$ for all $i\in\{1,2,\ldots,N\}$ and
\[
-J_{U_i}\ge (I(x_i)-\varepsilon)\wedge \varepsilon^{-1}\quad\mbox{ for all }i=1,2,\ldots,N.
\]
Set in addition $U_0:=K^c$. 
We have that $C\subset \cup_{i=0}^N U_i$.  
Since $J$ is weakly maxitive, we have
\[
-J_C \ge \wedge_{i=0}^N (-J_{U_i})\ge \wedge_{i=0}^N (I(x_i)-\varepsilon)\wedge \varepsilon^{-1} 
\ge \left(\underset{x\in C}\inf I(x) -\varepsilon\right)\wedge \varepsilon^{-1}.
\]
Letting $\varepsilon \downarrow 0$, we obtain the result. 
\end{proof}

\section{Integral representation of maxitive monetary risk measures}\label{sec:representation} 
A \emph{monetary risk measure}\footnote{Here, we use the terminology of \cite{follmer} up to a sign change.} is a function
 $\phi\colon B(E)\to[-\infty,\infty]$ 
satisfying:
\begin{itemize}
\item[(N)] Normalization:  $\phi(0)=0$,
\item[(M)] Monotonicity: $\phi(f)\le\phi(g)$ whenever $f\le g$,
\item[(T)] Translation invariance: $\phi(f+c)=\phi(f)+c$ for all~$c\in\R$. 
\end{itemize} 
We say that a monetary risk measure $\phi$ is \emph{maxitive} if it satifies 
$$\phi(f\vee g)\le \phi(f)\vee \phi(g)\quad \mbox{ for all }f,g\in B(E).$$ 
 Due to (b2)--(b4) in Section \ref{sec:preliminaries} the convex integral \eqref{eq:maxitiveInt} is a monetary risk measure which is maxitive if the corresponding concentration is maxitive.  
In the following, we focus on the converse direction and analyze when a monetary risk measure can be represented as a convex integral.     
Denote by $\bar{B}(E)$ the set of all $f\in B(E)$ such that $f$ is bounded from above. 
The following result was obtained in a slightly different setting in \cite[Corollary 6]{cattaneo}. 
We provide a short proof in the present setting.    
\begin{theorem}\label{thm:repI}
Let $\phi\colon B(E)\to[-\infty,\infty]$ be a maxitive monetary risk measure, and $J_A=\phi(-\infty 1_{A^c})$ for all $A\in\mathcal{B}(A)$. 
Then,  $J$ is a maxitive concentration, and 
\[
\phi(f)=\phi_J(f)
\]
for all $f\in \bar{B}(E)$.
\end{theorem}
\begin{proof}
For every function $f\colon E\to\R\cup\{-\infty\}$ that is bounded from above and every set $A\subset E$, we define
\[
\bar{\phi}(f):=\underset{g\in B(E)\colon f\le g}\inf\phi(g),\quad
\mbox{ and }\quad 
\bar{J}_A:=\bar{\phi}(-\infty 1_{A^c}).
\]
Inpection shows that $\bar{\phi}$ is finitely maxitive and additively homogeneous in the sense of \cite{cattaneo}. 
Then, by \cite[Corollary 6]{cattaneo}, we have that
\[
\bar{\phi}(f)=\underset{c\in\R}\sup\{c + \bar{J}_{\{f\ge c\}}\}
\]
for all $f$ that are bounded from above. 
In particular, for $f\in B(E)$, we have 
\[
{\phi}(f)=\bar{\phi}(f)= \underset{c\in\R}\sup\{c + \bar{J}_{\{f\ge c\}}\}\\
=\underset{c\in\R}\sup\{c + {J}_{\{f\ge c\}}\}=\phi_J(f).
\]
This completes the proof. 
\end{proof}

By relaxing the maxitivity condition it is still possible to represent a monetary risk measure on continuous functions. 
The following notion was introduced in \cite{kupper2}. 
\begin{definition}
A monetary risk measure $\phi\colon B(E)\to [-\infty,\infty]$ is said to be weakly maxitive if 
\[
\phi(f)\le \vee_{i=1}^N \phi(g_i)\mbox{ for all $f\in {U}(E)$, $g_1,g_2,\ldots,g_N\in {L}(E)$, $N\in\N$, such that $f\le \vee_{i=1}^N g_i$.}
\]
\end{definition}

Define ${C}(E)={U}(E)\cap {L}(E)$ and $\bar{C}(E)={C}(E)\cap \bar{B}(E)$.   
The following result was shown in  \cite[Theorem 4.2]{kupper2}.\footnote{To apply \cite[Theorem 4.2]{kupper2} here we consider the trivial preorder, i.e. $x\le y$ whenever $x=y$.} 
\begin{theorem}\label{thm:repIII}
Let $\phi\colon B(E)\to[-\infty,\infty]$ be a weakly maxitive monetary risk measure, and $J_A=\phi(-\infty 1_{A^c})$ for all $A\in\mathcal{B}(A)$. 
Then, $J$ is a weakly maxitive concentration and
\[
\phi(f)=\phi_J(f)
\]
for all~$f\in \bar{C}(E)$.  
\end{theorem}

Given a function $\phi\colon B(E)\to[-\infty,\infty]$, we define the sets 
\[
B_{\phi}(E):=\left\{f\in B(E)\colon \mbox{there exists }t>1\mbox{ such that }\phi(t f)<\infty\right\},
\]
\[
C_{\phi}(E):=\left\{f\in C(E)\colon \mbox{there exists }t>1\mbox{ such that }\phi(t f)<\infty\right\}.
\] 
\begin{lemma}\label{lem:tail}
Let $\phi\colon B(E)\to[-\infty,\infty]$ be a monetary risk measure. 
If $f\in B_{\phi}(E)$, then
\[
\underset{m\to\infty}\lim{\phi}\left(f 1_{\{f\ge m\}}-\infty 1_{\{f<m\}}\right)=\underset{m\to\infty}\lim{\phi}\left(f 1_{\{f> m\}}-\infty 1_{\{f\le m\}}\right)=-\infty.
\]
\end{lemma}
\begin{proof}
Suppose that $f\in B_{\phi}(E)$, and take $t>1$ such that $\phi(t f)<\infty$.  
Fix $m\in\N$ and define $g=\exp(f-m)$. 
Then, by translation invariance and monotonicity, we have
\begin{align*}
-m + \phi\left(f 1_{\{f\ge m\}}-\infty 1_{\{f<m\}}\right)&=\phi\left(-m  + f 1_{\{f\ge m\}}-\infty 1_{\{f<m\}}\right)\\
&=\phi\left(\log(g)1_{\{g\ge 1\}} - \infty 1_{\{g< 1\}}\right)\\
&\le\phi\left(\log(g^t)\right)\\
&=\phi\left(t(f-m)\right)\\
&=-mt + \phi\left(t f\right).
\end{align*}
Therefore, it follows that
\begin{align*}
\underset{m\to\infty}\lim{\phi}\left(f 1_{\{f> m\}}-\infty 1_{\{f\le m\}}\right)&\le 
\underset{m\to\infty}\lim{\phi}\left(f 1_{\{f\ge m\}}-\infty 1_{\{f<m\}}\right)\\
&\le \underset{m\to\infty}\lim\Big(m(1-t) + \phi\left(t f\right)\Big)
=-\infty, 
\end{align*}
where the latter limit is $-\infty$ since $t>1$ and $\phi\left(t f\right)<-\infty$. This completes the proof. 
\end{proof}

We next extend Theorem~\ref{thm:repI} and Theorem~\ref{thm:repIII} to unbounded functions as follows.
\begin{theorem}\label{thm:repII}
Let $\phi\colon B(E)\to[-\infty,\infty]$ be a monetary risk measure, and the concentration $J_A=\phi(-\infty 1_{A^c})$ for all $A\in\mathcal{B}(E)$. Then,
\begin{enumerate}
\item if $\phi$ is maxitive, then $\phi(f)=\phi_J(f)$ for all $f\in B_{\phi}(E)$,
\item if $\phi$ is weakly maxitive, then $\phi(f)=\phi_J(f)$ for all $f\in C_{\phi}(E)$.
\end{enumerate}
\end{theorem}
\begin{proof}
We prove (2). 
Suppose that $\phi$ is weakly maxitive. 
Fix $f\in C_{\phi}(E)$, and $n\in\N$.  
Since
$$f\le (f \wedge n)\vee \left(f1_{\{f> n\}}-\infty 1_{\{f\le n\}}\right).$$
Since $f$ is upper semicontinuous, and the functions in the maximum on the right hand side are lower semicontinuous,  
 it follows from the weak maxitivity of $\phi$ that
\begin{align*}
{\phi}(f)&\le {\phi}\left(f \wedge n\right)\vee {\phi}\left(f 1_{\{f> n\}}-\infty 1_{\{f\le n\}}\right)\\
&={\phi}_J\left(f \wedge n\right)\vee {\phi}\left(f 1_{\{f> n\}}-\infty 1_{\{f\le n\}}\right).
\end{align*}
where we have applied that ${\phi}\left(f \wedge n\right)={\phi}_J\left(f \wedge n\right)$ by Theorem~\ref{thm:repI}. 
On the other hand, 
\[
\phi(f)\ge \phi(f \wedge n)=\phi_J(f \wedge n).
\]
Both things together yield 
\[
\phi_J(f \wedge n)\le \phi(f)\le {\phi}_J\left(f \wedge n\right)\vee {\phi}\left(f1_{\{f> n\}}-\infty 1_{\{f\le n\}}\right).
\]
We have $\lim_{n\to\infty}{\phi}\left(f1_{\{f> n\}}-\infty 1_{\{f\le n\}}\right)=-\infty$ by Lemma \ref{lem:tail}. 
Then, by letting $n\to \infty$, we get that $\phi(f)=\phi_J(f)$.  
\end{proof}

\section{Main result}\label{sec:main}
Throughout this section we consider two monetary risk measures $\underline{\phi},\overline{\phi}\colon B(E)\to [-\infty,\infty]$ which satisfy the following.
\begin{assumption}\label{ass:jointMax}\mbox{}\newline\vspace{-0.5cm}
\begin{enumerate}
\item For every $f\in B(E)$, $\underline{\phi}(f)\le\overline{\phi}(f)$, 
\item $\overline{\phi}$ is weakly maxitive,
\item
$
\underline{\phi}(f)\le \underline{\phi}(g_1)\vee\Big(\vee_{i=2}^N\overline{\phi}(g_i)\Big)
$ 
for all $f\in {U}(E)$, and $g_1,g_2,\ldots,g_N\in {L}(E)$, $N\in\mathbb{N}$, such that $f\le \vee_{i=1}^N g_i$.\end{enumerate}
\end{assumption} 
\begin{remark}
 Assumption \ref{ass:jointMax} covers the case of a single (weakly) maxitive monetary risk measure $\phi$ by taking $\underline{\phi}:=\overline{\phi}:=\phi$. 
 In that case, (1)--(3) are automatically satisfied. 
In the application in next section, we will deal with lower/upper large deviations bounds, this is the reason why we consider a pair of monetary risk measures rather than a single one.  
\end{remark}
We consider the concentrations $\underline{J},\overline{J}\colon \mathcal{B}(E)\to [-\infty,0]$ given by $\underline{J}_A=\underline{\phi}(-\infty 1_{A^c})$ and $\overline{J}_A=\overline{\phi}(-\infty 1_{A^c})$. 
In addition, we denote by $\underline{I},\overline{I}$ the respective minimal rate functions defined as in \eqref{eq:minRate}. 
As in \cite{kupper}, we introduce the LDP and LP for monetary risk measures.
\begin{definition}
Suppose that $I\colon E\to[0,\infty]$ is a rate function. 
\begin{itemize}
\item We say that the pair $\underline{\phi},\overline{\phi}$ satisfies the \emph{large deviation principle} (LDP) with rate function $I(\cdot)$ if 
\[
-\underset{x\in {\rm int}(A)}\inf I(x)\le \underline{J}_A \le \overline{J}_A \le -\underset{x\in {\rm cl}(A)}\inf I(x)\quad\mbox{ for all }A\in \mathcal{B}(E). 
\]
\item We say that the pair $\underline{\phi},\overline{\phi}$ satisfies the \emph{Laplace principle} (LDP) with rate function 
\[
\underline{\phi}(f)=\overline{\phi}(f)=\underset{x\in E}\sup\{f(x)-I(x)\}\quad\mbox{ for all }f\in C_{\overline{\phi}}(E).
\] 
\end{itemize}
\end{definition} 
\begin{remark}
In \cite[Proposition 5.2]{kupper} it is proven the equivalence between LDP and LP under the hypothesis that $E$ is normal (the normality is only needed to prove that LP implies LPD).  
Notice, that in \cite{kupper} it is defined the LP on $C_b(E)$. 
Here, we consider the larger space $C_{\overline{\phi}}(E)$. 
That LDP implies LP as defined above is an easy consequence of the duality bounds proven in \cite{kupper2} (see Theorem \ref{thm:dualityBounds}). We give the argument in the Appendix A below. 
\end{remark} 
In the following, let $\mathcal{H}$ be a distinguished nonempty set of continuous real-valued functions on $E$.     
We define the \emph{conjugate} $\overline{\phi}_{\mathcal{H}}^\ast\colon E\to[-\infty,\infty]$ of $\overline{\phi}$ with respect to $\mathcal{H}$ as
$$
\overline{\phi}_{\mathcal{H}}^\ast(x):=\underset{f\in \mathcal{H}}\sup\{f(x)-\overline{\phi}(f)\}. 
$$
\begin{definition} 
We say that $x\in E$ is an $\mathcal{H}$-\emph{exposed point} of $\overline{\phi}_{\mathcal{H}}^\ast$ if there exists $f\in \mathcal{H}$ such that
\[
f(y)-\overline{\phi}_{\mathcal{H}}^\ast(y)<f(x)-\overline{\phi}_{\mathcal{H}}^\ast(x)\quad\mbox{ for all }y\neq x. 
\]
In that case, we say that $f$ is an \emph{exposing function} for $x$. 
We denote by $\mathscr{E}$ the set of all $\mathcal{H}$-exposed  points $x\in E$ of $\overline{\phi}_{\mathcal{H}}^\ast$ which admit an exposing function $f\in \mathcal{H}$ such that 
\begin{equation}\label{eq:niceExp}
\overline{\phi}(f)=\underline{\phi}(f),\quad\mbox{ and }\quad f\in B_{\overline{\phi}}(E).
\end{equation}
\end{definition} 

\begin{remark}
If $x$ is an $\mathcal{H}$-{exposed point}, then $\overline{\phi}_{\mathcal{H}}^\ast(x)<\infty$ and $f(x)>-\infty$. 
\end{remark}

We present the main result of this section.
\begin{theorem}\label{thm:main}
Suppose that $\overline{J}$ is tight.   
Then:
\begin{enumerate}[(i)]
\item For all $x\in\mathscr{E}$, $\underline{I}(x)=\overline{I}(x)=\overline{\phi}^\ast_{\mathcal{H}}(x)$.  
\item For every closed set $C\subset E$, we have the upper bound
\[
\overline{J}_C\le -\underset{y\in C}\inf \overline{\phi}_{\mathcal{H}}^\ast(x).
\]
\item For every open set $O\subset E$, we have the lower bound
\[
 -\underset{y\in O\cap\mathscr{E}}\inf \overline{\phi}_{\mathcal{H}}^\ast(x)\le \underline{J}_O.
\]
\item If, moreover,  
\begin{equation}\label{eq:condition}
\underset{x\in O}\inf \overline{\phi}_{\mathcal{H}}^\ast(x)=\underset{x\in O\cap\mathscr{E}}\inf \overline{\phi}_{\mathcal{H}}^\ast(x)
\quad\mbox{ for every } O\subset E\mbox{ open,}
\end{equation} 
then the pair $\underline{\phi},\overline{\phi}$ verifies the LDP and LP with rate function $\overline{\phi}_{\mathcal{H}}^\ast$.  
\end{enumerate}
\end{theorem}


\begin{remark}
Given a sequence of $E$-valued random variables $(X_n)_{n\in\N}$ defined in a probability space $(\Omega,\mathcal{F},\mathbb{P})$, we define the upper/lower asymptotic entropies by
\[
\underline{\psi}(f):=\underset{n\to\infty}\liminf \frac{1}{n}\log\mathbb{E}_n[e^{n f(X_n)}],
\quad
\overline{\psi}(f):=\underset{n\to\infty}\limsup \frac{1}{n}\log\mathbb{E}_n[e^{n f(X_n)}].
\]    
We prove in Section \ref{sec:LDsublinearExp} below that $\underline{\psi},\overline{\psi}$ satisfy  Assumption \ref{ass:jointMax}, in addition, we also have that the respective concentrations $\underline{J},\overline{J}$ are given by
\[
\underline{J}_A=\underset{n\to\infty}\liminf \frac{1}{n}\log\mathbb{P}(X_n\in A),\quad
\overline{J}_A=\underset{n\to\infty}\limsup \frac{1}{n}\log\mathbb{P}(X_n\in A).
\] 
In addition, the tightness of $\overline{J}$ is exactly the exponential tightness of the sequence $(X_n)_{n\in\N}$. 
Then, we have that Theorem \ref{thm:linear} in the introduction is a direct consequence of  Theorem \ref{thm:main}.
\end{remark}

In order to prove Theorem \ref{thm:main}, we need some preliminary results. 

\begin{lemma}\label{lem:rateMinKappa}
$\overline{\phi}_{\mathcal{H}}^\ast(x)\le \overline{I}(x)\le \underline{I}(x)$ for all~$x\in E$.
\end{lemma}
\begin{proof}
Given $x\in E$, we know from their respective definitions that $\overline{I}(x)\le \underline{I}(x)$. 
Fix now $f\in\mathcal{H}$ and $n\in\N$. 
Since $\overline{\phi}$ is weakly maxitive and $f\wedge n \in \bar{C}(E)$, by Theorem \ref{thm:repIII} we have
\[
{\phi}_{\overline{J}}(f\wedge n)=\overline{\phi}(f\wedge n).
\]  
Then, by monotonicity, we have
\[
{\phi}_{\overline{J}}(f\wedge n)=\overline{\phi}(f\wedge n)\le \overline{\phi}(f). 
\]
Due to (b5), letting $n\to \infty$ results in    
\[
{\phi}_{\overline{J}}(f)\le \overline{\phi}(f). 
\]
Thus, we have
\[
f(x) - \overline{\phi}(f) \le f(x) - {\phi}_{\overline{J}}(f) \le \overline{I}(x).
\]
Since $f\in\mathcal{H}$ was arbitrary, it follows that
\[
\overline{\phi}_{\mathcal{H}}^\ast(x)=\underset{f\in\mathcal{H}}\sup\{f(x)-\overline{\phi}(f)\}\le \overline{I}(x).
\]
The proof is complete.
\end{proof}

\begin{lemma}\label{lem:unifExposed}
Let $K\subset E$ be compact and $x\in E$ an $\mathcal{H}$-{exposed point} of $\overline{\phi}_{\mathcal{H}}^\ast$ with exposing function $f\in \mathcal{H}$. 
Then, for every open set $U\subset E$ such that $x\in U$ there exists an open set $W\subset E$  such that 
\begin{enumerate}
\item $K\cap U^c \subset W$,
\item $\underset{y\in {\rm cl}(W)}\sup\left\{f(y)-\overline{\phi}_{\mathcal{H}}^\ast(y)\right\}<f(x)-\overline{\phi}_{\mathcal{H}}^\ast(x)$.
\end{enumerate}

\end{lemma}
\begin{proof}
For each $\varepsilon>0$, define
\[
V_\varepsilon:=\big\{y\in E \colon f(y)-\overline{\phi}_{\mathcal{H}}^\ast(y)+\varepsilon < f(x)-\overline{\phi}_{\mathcal{H}}^\ast(x) \big\}.
\]
Since the mapping $y\mapsto f(y)-\overline{\phi}_{\mathcal{H}}^\ast(y)$ is upper semicontinuous, we have that $V_\varepsilon$ is open.   
We claim that there exists $\varepsilon>0$ such that $K\cap U^c\subset V_\varepsilon$.  
Indeed, by contradiction assume that for every $\varepsilon>0$ we can pick up $y_\varepsilon\in K\cap U^c$ such that
\[
 f(y_\varepsilon)-\overline{\phi}_{\mathcal{H}}^\ast(y_\varepsilon)+\varepsilon\ge f(x)-\overline{\phi}_{\mathcal{H}}^\ast(x).
\] 
Then, $(y_\varepsilon)_{\varepsilon>0}$ is a net in the compact set $K$.\footnote{Here, $\{\varepsilon\colon \varepsilon>0\}$ is regarded as downwards directed set.} 
We can take a subnet $(y_{\varepsilon_\alpha})$ such that $y_{\varepsilon_\alpha}\to y\in K\cap U^c $. 
Taking the limsup   on $\alpha$, and using that $y\mapsto f(y)-\overline{\phi}_{\mathcal{H}}^\ast(y)$ is upper semicontinuous, we get
\[
 f(y)-\overline{\phi}_{\mathcal{H}}^\ast(y) \ge f(x)-\overline{\phi}_{\mathcal{H}}^\ast(x).
\]  
Besides, we have that $y\neq x$ as $y\in U^c$ and $x\in U$.  
This contradicts that $x$ is an $\mathcal{H}$-exposed point. 
We have that $K\cap U^c$ is a compact set contained in the open set $V_\varepsilon$. 
 Since $E$ regular, we can find an open set $W\subset E$ such that $K\cap U^c \subset{\rm cl}(W)\subset V_\varepsilon$. 
Finally, 
the set $W$ meets the required conditions.   
\end{proof}

\begin{proposition}\label{eq:concentrationExp}
Suppose that $x\in\mathscr{E}$ and $f\in\mathcal{H}$ is an exposing function for $x$ satisfying~\eqref{eq:niceExp}. 
If $\overline{J}$ is tight, then for every open set $U\subset E$ with $x\in U$ it holds  
\[
\overline{\phi}(f 1_U - \infty 1_{U^c})=\overline{\phi}(f)=\underline{\phi}(f)=\underline{\phi}(f 1_U - \infty 1_{U^c}).
\]
\end{proposition}
\begin{proof}
Let $f\in \mathcal{H}$ be an exposing function for $x$ satisfying~\eqref{eq:niceExp}. 
Fix $n\in\N$.  
Since $\overline{J}$ is tight, there exists a compact set $K\subset E$ such that 
\begin{equation}\label{eq:tightness}
\overline{J}_{K^c}<-2 n. 
\end{equation}
Due to Lemma~\ref{lem:unifExposed}, we can find an open set $W\subset E$ such that  
\begin{equation}
\label{eq:unifExposed}
K\cap U^c\subset W,\quad 
\underset{y\in {\rm cl}(W)}\sup\left\{f(y)-\overline{\phi}_{\mathcal{H}}^\ast(y)\right\}<f(x)-\overline{\phi}_{\mathcal{H}}^\ast(x).
\end{equation}
Since 
\begin{align*}
E&=K\cup K^c\\
&\subset  (K\cap U)\cup (K\cap U^c)\cup \Big(K^c\cap\{f < n+1\}\Big)\cup \Big(K^c\cap\{f > n\}\Big)\\
&\subset   U \cup W \cup \Big(K^c\cap\{f < n+1\}\Big)\cup \{f > n\},
\end{align*}
 due to Assumption \ref{ass:jointMax} we have
\begin{align}
\overline{\phi}(f)&=\underline{\phi}(f)\nonumber\\
&\le  \underline{\phi}\left(f 1_U - \infty 1_{U^c}\right)
\vee \overline{\phi}\left(f 1_W - \infty 1_{W^c}\right)\vee \overline{\phi}\left((n+1) 1_{K^c} - \infty 1_K\right)
\vee \overline{\phi}\left(f 1_{\{f > n\}} - \infty 1_{\{f\le n\}}\right).
 \label{eq:mainIne}
\end{align}
In addition, by the definition of $\overline{\phi}^\ast_{\mathcal{H}}$ we have 
\begin{equation}\label{eq:ineBelow}
\overline{\phi}(f)\ge f(x)-\overline{\phi}^\ast_{\mathcal{H}}(x).
\end{equation}
On the other hand, $\overline{J}$ is weakly maxitive and tight. 
Then, due to Theorem~\ref{thm:repII} and Proposition~\ref{prop:Bryc} we have
\begin{align*}
\overline{\phi}\left(f 1_{W} -\infty 1_{W^c}\right)&\le \overline{\phi}\left(f 1_{{\rm cl}(W)} - \infty 1_{{\rm cl}(W)^c} \right)\\
&=\phi_{\overline{J}}\left(f 1_{{\rm cl}(W)} - \infty 1_{{\rm cl}(W)^c} \right)\\
&\le \underset{y\in {\rm cl}(W)}\sup\{f(y)-\overline{I}(y)\}\\
&\le \underset{y\in {\rm cl}(W)}\sup\{f(y)-\overline{\phi}^\ast_{\mathcal{H}}(y)\}\\
&<f(x)-\overline{\phi}^\ast_{\mathcal{H}}(x),
\end{align*}
where we have used that $\overline{\phi}^\ast_{\mathcal{H}}(y)\le \overline{I}(y)$ due to Lemma \ref{lem:rateMinKappa} in the second inequality and \eqref{eq:unifExposed} in the third inequality. 
The last inequality is strict, then, in view of \eqref{eq:ineBelow}, we can drop the second member of the maximum in \eqref{eq:mainIne}, obtaining
\begin{equation}\label{eq:mainIne2}
\overline{\phi}(f)\le  \underline{\phi}\left(f 1_U - \infty 1_{U^c}\right)
\vee  \overline{\phi}\left((n+1) 1_{K^c} - \infty 1_K\right)
\vee \overline{\phi}\left(f 1_{\{f > n\}} - \infty 1_{\{f\le n\}}\right).
\end{equation}
By monotonicity and translation invariance, we get
\begin{align*}
\overline{\phi}\left((n+1) 1_{K^c} - \infty 1_K\right) &=\overline{\phi}\left(- \infty 1_K\right)+n+1\\
&= \overline{J}_{K^c} + n+1\\
& \le -2 n + n + 1=-n + 1,
\end{align*}
where we have used \eqref{eq:tightness} in the last inequality. Therefore,
\begin{equation}
\label{eq:limitI}
\underset{n\to\infty}\lim \overline{\phi}\left((n+1) 1_{K^c} - \infty 1_K\right)=-\infty.
\end{equation}
On the other hand,  $f\in B_{\overline{\phi}}(E)$ by \eqref{eq:niceExp}. Then, applying Lemma~\ref{lem:tail} we have 
\begin{equation}
\label{eq:limitII}
\lim_{n\to\infty}\overline{\phi}\left(f 1_{\{f > n\}} - \infty 1_{\{f\le n\}}\right)=-\infty.
\end{equation}
Consequently, letting $n\to\infty$ in \eqref{eq:mainIne2} results in
\[
\overline{\phi}(f)\le \underline{\phi}\left(f 1_U - \infty 1_{U^c}\right).
\] 
Finally, by monotonicity we have 
$$\underline{\phi}\left(f 1_U - \infty 1_{U^c}\right)\le \overline{\phi}\left(f 1_U - \infty 1_{U^c}\right)\le  \overline{\phi}(f)\le \underline{\phi}\left(f 1_U - \infty 1_{U^c}\right),$$ 
and the desired equalities follow. 
\end{proof}

We now turn to the proof of Theorem \ref{thm:main}.

\begin{proof} 
Let $x\in\mathscr{E}$, and take an exposing function $f\in\mathcal{H}$ for $x$ satisfying  \eqref{eq:niceExp}.    
Fix an open neighborhood $U$ of $x$ and $\varepsilon>0$.   
Since $f$ is upper semicontinuous, we can find an open neighborhood $V\subset  U$ of $x$ such that 
\begin{equation}\label{eq:continuity}
f(y)<f(x)+\varepsilon\quad\mbox{ for all }y\in V.
\end{equation}
Due to Proposition~\ref{eq:concentrationExp},  
\begin{align*}
\overline{\phi}(f)&=\underline{\phi}(f 1_{V} - \infty 1_{V^c})\\
&\le \underline{\phi}\big((f(x)+\varepsilon)1_V - \infty 1_{V^c}\big) 
=\underline{J}_V + f(x)+\varepsilon\le \underline{J}_U + f(x)+\varepsilon.  
\end{align*}
On the other hand, from the definition of $\overline{\phi}^\ast_{\mathcal{H}}$ we have that 
\[
\overline{\phi}(f)\ge f(x)-\overline{\phi}^\ast_{\mathcal{H}}(x).
\]
Combining both things, we get
\[
-\overline{\phi}^\ast_{\mathcal{H}}(x) \le \underline{J}_U + \varepsilon. 
\] 
Letting $\varepsilon \downarrow 0$, we obtain
\[
-\overline{\phi}^\ast_{\mathcal{H}}(x) \le \underline{J}_U. 
\] 
Since $U$ was arbitrary, it follows from Proposition~\ref{prop:representationRate} that
\[
\overline{\phi}^\ast_{\mathcal{H}}(x) \ge \underline{I}(x).
\]  
Finally, we have by Lemma~\ref{lem:rateMinKappa} that $\overline{\phi}^\ast_{\mathcal{H}}(x)\le\overline{I}(x)\le \underline{I}(x)$, obtaining (i).

 Suppose that $C\subset E$ is closed. 
Since $\overline{J}$ is tight and weakly maxitive, by Proposition~\ref{prop:Bryc} we have that
\[
\overline{J}_C\le -\underset{x\in C}\inf \overline{I}(x) \le -\underset{x\in C}\inf \overline{\phi}^\ast_{\mathcal{H}}(x),
\]
where we have used that $\overline{\phi}^\ast_{\mathcal{H}}(x) \le \overline{I}(x)$. This proves (ii).  

Suppose now that $O\subset E$ is open. 
It follows from Proposition \ref{prop:lowerBounds} that
\begin{align*}
\underline{J}_O &\ge -\underset{x\in O}\inf \underline{I}(x)\\
&\ge -\underset{x\in O\cap\mathscr{E}}\inf \underline{I}(x)\\ 
&= -\underset{x\in O\cap\mathscr{E}}\inf \overline{\phi}^\ast_{\mathcal{H}}(x) , 
\end{align*} 
where we have used that $\underline{I}(x)=\overline{\phi}^\ast_{\mathcal{H}}(x) $ for all $x\in\mathscr{E}$ by Proposition (i).~Then, (iii) follows. 

Finally, we obtain that the pair $\underline{\phi},\overline{\phi}$ satisfies the LDP with rate function $\overline{\phi}^\ast_{\mathcal{H}}$ as a consequence of (ii) and (iii) taking into account \eqref{eq:condition}.  
In turn, due to Proposition~\ref{prop:LDPimpliesLP} the pair $\underline{\phi},\overline{\phi}$ also satisfies the LP with rate function $\overline{\phi}^\ast_{\mathcal{H}}$. 
\end{proof}

\begin{examples}\label{examples:linear}
\mbox{}\newline\vspace{-0.5cm}
\begin{enumerate}
 \item \textbf{Topological vector spaces}: Every topological vector space is regular. 
 Then, Theorem \ref{thm:main} applies to (Hausdorff) topological  vector spaces. 
 In that case, we can consider $\mathcal{H}:=E^\ast$, the set of all linear continuous real-valued functions on $E$. 
 Define $\Lambda:=\overline{\phi}|_{E^\ast}$.  
 Since $\overline{\phi}$ is weakly maxitive due to Assumption \ref{ass:jointMax}, we have that $\overline{\phi}$ is maxitive on the set $C(E)$. 
It follows from \cite[Proposition 2.1]{kupper} that $\overline{\phi}$ is convex on $C(E)$ and, in particular, $\Lambda$ is a convex function.  
 Then, we have that $\Lambda^\ast:=\phi^\ast_{E^\ast}$ is the convex conjugate of $\Lambda$.     
 \item \textbf{Finite dimension}:
Suppose now that $E=\R^d$ and $\mathcal{H}=(\R^d)^\ast=\R^d$.  
Consider the conditions below.  
\begin{enumerate}
\item $\Lambda(y)=\underline{\phi}(y)=\overline{\phi}(y)$ for $y\in \R^d$,
\item $0$ belongs to the topological interior of $\{y\in\R^d\colon \overline{\phi}(y)<\infty\}$,
\item $\Lambda$ is lower semicontinuous,
\item $\Lambda$ is \emph{essentially smooth} in the sense of \cite[Definition 2.3.5]{dembo}.
\end{enumerate}
Then, under the conditions (a)--(c) above, the condition \eqref{eq:condition}  in Theorem \ref{thm:main} holds. 
This is proved by following word by word the argumentation in the proof of (c) in \cite[Theorem 2.3.6]{dembo}. 
We conclude that, in the present situation, we can replace \eqref{eq:condition} in Theorem \ref{thm:main} by the conditions (a)--(c) above.

\item \textbf{Exposing families of functions}: A family $(f_a)_{a\in E}$ of functions in $C_{\overline{\phi}}$ is said to be an \emph{exposing family} for $\underline{\phi},\overline{\phi}$ if for every $x\in E$ we have $\overline{\phi}(f_x)=\underline{\phi}(f_x)=0$ and
$$f_x(y)<\underset{a\in E}\sup f_a(y)\quad\mbox{ for all }y\neq x.$$ 
If $\overline{J}$ is tight and $\mathcal{H}=\{f_a\colon a\in E\}$ for an exposing family $(f_a)_{a\in E}$, then $\mathscr{E}=E$, 
\[
\overline{I}(x)=\underline{I}(x)=\overline{\phi}^\ast_{\mathcal{H}}(x)=\underset{a\in E}\sup f_a(x)\quad\mbox{ for all }x\in E,
\]
and the pair $\underline{\phi},\overline{\phi}$ satisfies the LDP and LP with rate function $\overline{\phi}^\ast_{\mathcal{H}}$. 
This can be proven as follows. 
First, since $\overline{\phi}(f_a)=0$ for all $a\in E$, it follows that
\[
\overline{\phi}^\ast_{\mathcal{H}}(x)=\sup_{a\in E} f_a(x)\quad\mbox{for all }x\in E.
\]
Then, given $x\in E$, for each $y\neq x$ we have
\[
\overline{\phi}^\ast_{\mathcal{H}}(y)-f_x(y)=\sup_{a\in E} f_a(y)-f_x(y)>0.
\]
On the other hand, 
\[
\overline{\phi}^\ast_{\mathcal{H}}(x)-f_x(x)=\sup_{a\in E} f_a(x)-f_x(x)\le 0.
\]
It follows that $x$ is $\mathcal{H}$-exposed. 
Since $x$ was arbitrary, we conclude that $\mathscr{E}=E$.  
Thus, the condition \eqref{eq:condition} is trivially satisfied and we get the conclusions as a consequence of Theorem~\ref{thm:main}. 
\end{enumerate}
\end{examples}

\section{Large deviation principle for sequences of sublinear expectations}\label{sec:LDsublinearExp}
We finally apply Theorem \ref{thm:main} to study large deviations for sequences of sublinear expectations.  
Denote by ${B}_+(E)$ the set of all Borel measurable functions $f\colon E\to [0,\infty)$.   
A function $\mathcal{E}\colon {B}_+(E)\to [0,\infty]$ is called a \emph{sublinear} expectation if:
\begin{enumerate}
\item $\mathcal{E}(c)=c$ for all constant $c\ge 0$,
\item $\mathcal{E}(f)\le\mathcal{E}(g)$ whenever $f\le g$,
\item $\mathcal{E}(f+g)\le \mathcal{E}(f)+\mathcal{E}(g)$,
\item $\mathcal{E}(a f)=a\mathcal{E}(f)$ for all constant $a\ge 0$.
\end{enumerate}
A functional which satisfies the properties (1)-(4) is also called upper expectation in robust statistics~\cite{huber}, upper coherent prevision in the theory of imprecise probabilities \cite{walley}, or (up to a sign change) coherent risk measure in mathematical finance \cite{delbaen}.
\begin{example}
Suppose that $X$ is an $E$-valued random variable defined on a probability space $(\Omega,\mathcal{F},\mathbb{P})$. 
Consider a nonempty set $\mathcal{P}$ of probability measures on $\mathcal{F}$. 
Then, the mapping $\mathcal{E}\colon {B}_+(E)\to [0,\infty]$  given by
\[
\mathcal{E}(f):=\underset{Q\in\mathcal{P}}\sup\mathbb{E}_Q[f(X)]
\] 
is a sublinear expectation.
\end{example}

In the following, we consider a sequence $(\mathcal{E}_n)_{n\in\N}$ of sublinear expectations. 
We define the lower/upper \emph{asymptotic entropies} $\underline{\psi},\overline{\psi}\colon B(E)\to\overline{\R}$ as
\[
\underline{\psi}(f):=\underset{n\to\infty}\liminf \frac{1}{n}\log\mathcal{E}_n(e^{n f}),
\quad
\overline{\psi}(f):=\underset{n\to\infty}\limsup \frac{1}{n}\log\mathcal{E}_n(e^{n f}).
\]
Straightforward inspection shows that $\underline{\psi}$ and $\overline{\psi}$ are monetary risk measures. 
The following lemma is well-known in large deviations theory and often referred to as the principle of the largest term; see e.g.~\cite[Lemma 1.2.15]{dembo} for \eqref{eq:largestTermUp} and \cite[Exercise
14.8]{rassoul} for \eqref{eq:largestTermLow}. 
\begin{lemma}\label{lem:largestTerm}
Suppose that $(a_n^1)_{n\in\N},(a_n^2)_{n\in\N},\ldots,(a_n^N)_{n\in\N}$ are $[0,\infty]$-valued sequences, then
\begin{equation}\label{eq:largestTermUp}
\underset{n\to\infty}\limsup \frac{1}{n}\log \sum_{i=1}^N a_n^i\le \vee^N_{i=1} \underset{n\to\infty}\limsup \frac{1}{n}\log a_n^i,
\end{equation}
\begin{equation}\label{eq:largestTermLow}
\underset{n\to\infty}\liminf \frac{1}{n}\log \sum_{i=1}^N a_n^i\le  \Big(\underset{n\to\infty}\liminf \frac{1}{n}\log a_n^1\Big) \vee\Big(\vee_{i=2}^N \underset{n\to\infty}\limsup \frac{1}{n}\log a_n^i\Big).   
\end{equation}
\end{lemma}  

As a consequence, we have. 
\begin{lemma}\label{lem:joinMax}
$\underline{\psi},\overline{\psi}$ satisfy Assumption \ref{ass:jointMax}.
\end{lemma}
\begin{proof}
We prove that
\[
\overline{\psi}(f)\le \vee_{i=1}^N\overline{\psi}(g_i)
\]
 for $f\in {U}(E)$, and $g_1,g_2,\ldots,g_n\in {L}(E)$ with $f\le \vee_{i=1}^n g_i$. 
 Indeed, as a consequence of Lemma \ref{lem:largestTerm}, we have 
\begin{align*}
\overline{\psi}(f)&\le \underset{n\to\infty}\limsup \frac{1}{n}\log\mathcal{E}_n\left(e^{n (\vee_{i=1}^N g_i)}\right)\\
&\le \underset{n\to\infty}\limsup\frac{1}{n}\log \mathcal{E}_n\left( \sum_{i=1}^N e^{n g_i}\right) \\
&\le \underset{n\to\infty}\limsup\frac{1}{n}\log \sum_{i=1}^N\mathcal{E}_n\left(  e^{n g_i}\right) \\
&\le \vee_{i=1}^N \underset{n\to\infty}\limsup\frac{1}{n}\log \mathcal{E}_n\left(  e^{n g_i}\right) \\
&=\vee_{i=1}^N\overline{\psi}(g_i).
\end{align*}
\end{proof}
Adopting the usual terminology of standard large deviations theory \cite{dembo}, we introduce the following: 
\begin{itemize}
\item We say that the sequence $(\mathcal{E}_n)_{n\in\N}$  is \emph{exponentially tight} if for every $n\in\N$ there exists a compact set $K\subset E$ such that 
\[
\underset{n\to\infty}\limsup \frac{1}{n}\log\mu_n(K^c)<-n.
\] 
Here, $\mu_n\colon\mathcal{B}(E)\to[0,1]$ is the capacity associated with $\mathcal{E}_n$ which is given by $\mu_n(A)=\mathcal{E}_n(1_A)$.  
\item We say that $(\mathcal{E}_n)_{n\in\N}$ satisfies the \emph{large deviation principle} (LDP) with rate function $I\colon E\to[0,\infty]$ if  
\begin{align*}
-\underset{x\in {\rm int}(A)}\inf I(x)\le &\liminf_{n\to\infty}\frac{1}{n}\log\mu_n(A)\\ 
\le &\limsup_{n\to\infty}\frac{1}{n}\log\mu_n(A) \le -\underset{x\in {\rm cl}(A)}\inf I(x)
\end{align*}
 for all $A\in\mathcal{B}(E)$.
\item  We say that $(\mathcal{E}_n)_{n\in\N}$ satisfies the \emph{Laplace principle} (LP) with rate function $I\colon E\to[0,\infty]$ if 
\begin{equation}\label{eq:LP2}
\underset{n\to\infty}\lim\frac{1}{n}\log\mathcal{E}_n(e^{n f})
=\underset{x\in E}\sup\{f(x)-I(x)\}
\end{equation}
for all $f\in C_{\overline{\psi}}(E)$.  
\end{itemize}

Notice that the exponential tightness of $(\mathcal{E}_n)_{n\in\N}$ is equivalent to the tightness of $\overline{J}$ in the sense of Definition \ref{def:tightness}. 

We turn next to the main result of this section. 
 Let $\mathcal{H}$ be a distinguished nonempty set of continuous real-valued functions on $E$.    	
 Denote by $\mathscr{E}$ the set of all $\mathcal{H}$-exposed points of $\overline{\psi}_{\mathcal{H}}^\ast$ which admit  an exposing function $f\in \mathcal{H}$ such that
\[
\underline{\psi}(f)=\overline{\psi}(f)=\underset{n\to\infty}\lim\frac{1}{n}\log\mathcal{E}_n(e^{n f})\quad\mbox{and}\quad f\in B_{\overline{\psi}}(E).  
\]
In view of Lemma \ref{lem:joinMax}, we have the following result as a direct consequence of Theorem \ref{thm:main}. 
\begin{theorem}\label{thm:sublinear}
Suppose that the sequence $(\mathcal{E}_n)_{n\in\N}$ of sublinear expectations is exponentially tight.  
Then: 
\begin{itemize}
\item[(i)] For every closed set $C\subset E$, we have the upper bound
\[
\limsup_{n\to\infty}\frac{1}{n}\log\mu_n(C)\le -\underset{y\in C}\inf \overline{\psi}_{\mathcal{H}}^\ast(x).
\]
\item[(ii)] For every open set $O\subset E$, we have the lower bound
\[
 -\underset{y\in O\cap\mathscr{E}}\inf \overline{\psi}_{\mathcal{H}}^\ast(x)\le \liminf_{n\to\infty}\frac{1}{n}\log\mu_n(O).
\]
\item[(iii)]
If, moreover,  
\begin{equation}\label{eq:richExpPoints}
\underset{x\in O}\inf \overline{\psi}_{\mathcal{H}}^\ast(x)=\underset{x\in O\cap{\mathscr{E}}}\inf \overline{\psi}_{\mathcal{H}}^\ast(x)\quad\mbox{for all }O\subset E\mbox{ open, }
\end{equation}
then $(\mathcal{E}_n)_{n\in\N}$ satisfies the LDP and LP with rate function $\overline{\psi}_{\mathcal{H}}^\ast$. 
\end{itemize}
\end{theorem}

In the special case when $\mathcal{H}=E^\ast$ is the dual space of a topological vector space $E$ Theorem \ref{thm:sublinear} above amounts to the well-known G\"{a}rtner-Ellis theorem; see \cite[Theorem 4.5.20]{dembo}, and \cite[Theorem 3.1]{tan} for a version for sublinear expectations.  
We have that Theorem \ref{thm:sublinear} above is more flexible as it allows for arbitrary choices for $\mathcal{H}$.     
In the following example we show a case where the LDP and LP follow from Theorem \ref{thm:sublinear} above, but it is not covered by G\"{a}rtner-Ellis theorem.  
\begin{example}\label{ex:classical}
For each $n\in\N$, let $X_n$ be a real-valued random variable with (centered) Laplace distribution of parameter $1/n$, i.e., $X_n$ has density $h_n(x)=\tfrac{n}{2}e^{-n|x|}$ for all $x\in\mathbb{R}$.\footnote{Equivalently,  $X_n$ is the different of two independent random variables with exponential distribution of parameter $n$.} 
Consider the lower/upper asymptotic entropies 
\[
\underline{\psi}(f)=\underset{n\to\infty}\liminf \frac{1}{n}\log\mathbb{E}_{\mathbb{P}}[e^{n f(X_n)}],
\quad
\overline{\psi}(f)=\underset{n\to\infty}\limsup \frac{1}{n}\log\mathbb{E}_{\mathbb{P}}[e^{n f(X_n)}].
\]
For every $m\in\N$, define the compact set $K_m:=[-m,m]$. 
Then, $\overline{J}_{K_m^c}=-m$. Hence, $(X_n)_{n\in\N}$ is exponentially tight. 
Consider $\mathcal{H}_1=\R^\ast=\R$ as in G\"{a}rtner-Ellis theorem \cite[Theorem 4.5.20]{dembo}.  
In that case, for every $y\in\mathcal{H}_1=\R$,
\[
\underline{\psi}(y)=\overline{\psi}(y)=\underset{n\to\infty}\lim\frac{1}{n}\log \int_{-\infty}^\infty\tfrac{n}{2} e^{n y x - n |x|} d x = \begin{cases}0, & \mbox{ if }|y|<1,\\ \infty, & \mbox{ if }|y|\ge 1.\end{cases} 
\]
Hence, 
\[
\overline{\psi}^\ast_{\mathcal{H}_1}(x)=\underset{y\in \R}\sup\{ y x - \overline{\psi}(y)\}=|x|.
\]
The only exposed point of $\overline{\psi}^\ast_{\mathcal{H}_1}$ is $0$. 
Then, for every open set $O\subset\R$ that does not contain the origin, the classical 
G\"{a}rtner-Ellis theorem only gives a trivial lower bound $-\inf_{y\in O\cap \{0\}}\overline{\psi}^\ast_{\mathcal{H}_1}(y)=-\infty$, and the condition~\eqref{eq:richExpPoints} is not satisfied. 

Now, consider the set of continuous functions 
\[
\mathcal{H}_2:=\left\{ f_a  \colon a\ge 0\right\},
\] 
where
$$f_a(x):=|a| - 2|x-a|.$$
 Direct verification yields $\underline{\psi}(f_a)=\overline{\psi}(f_a)=0$ for every $a\in\R$.
Consequently, we have
\[
\overline{\psi}^\ast_{\mathcal{H}_2}(x)=\underset{a\in\mathbb{R}}\sup f_a(x)=|x|.
\]
This shows that $\mathcal{H}_2$ is an exposing family as in Examples \ref{examples:linear}(3).  
Therefore, the set of all $\mathcal{H}_2$-exposed points is $\mathbb{R}$ and $\overline{\psi}^\ast_{\mathcal{H}_2}$ trivially verifies the condition~\eqref{eq:richExpPoints}.  
\begin{figure}[htb]
\centering
\includegraphics[width=8cm]{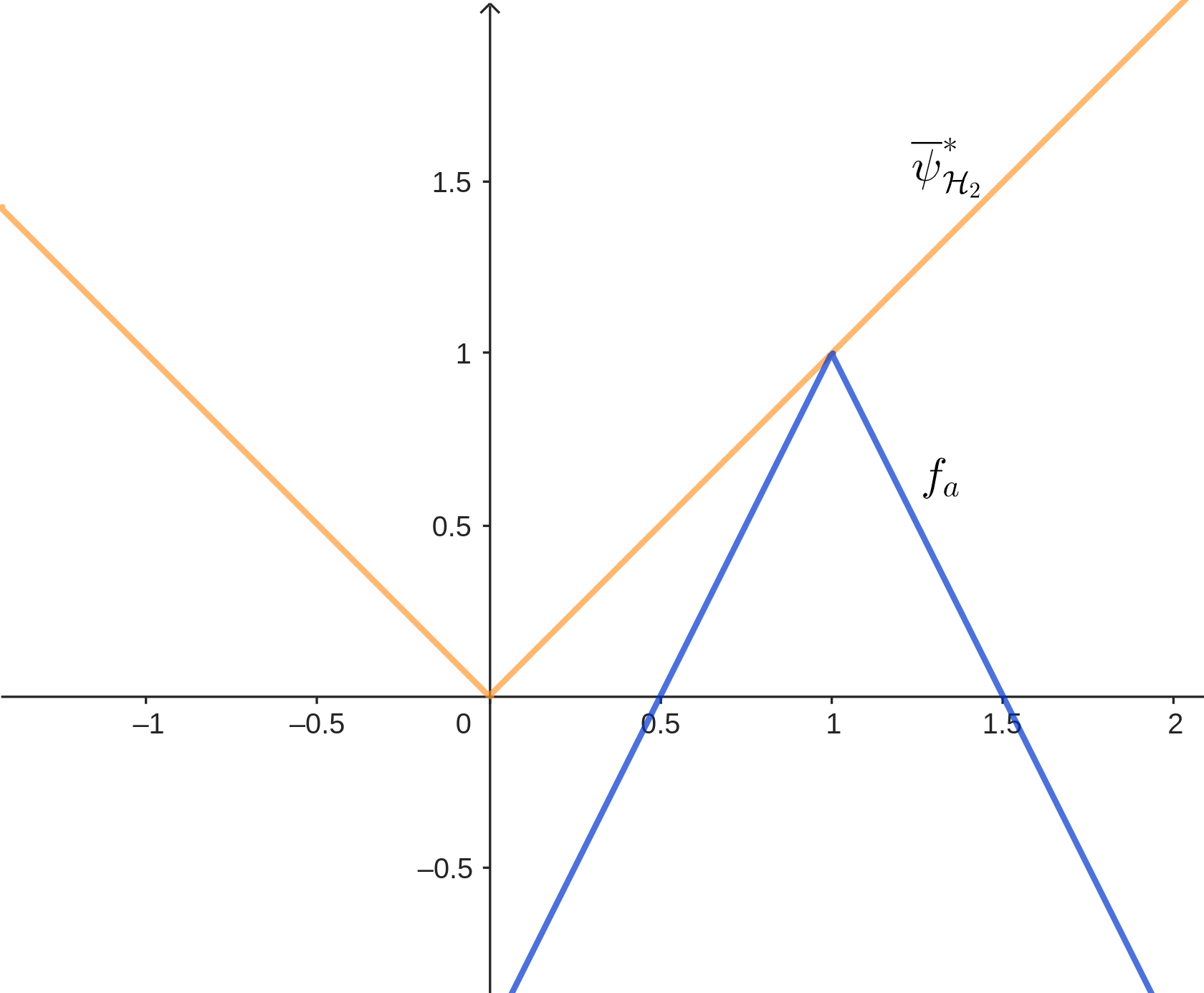}
\caption{Every point $a$ is exposed with exposing function $f_a$. In this case, we can see the exposing function for $a=1$.}
\end{figure}

Then, (iii) in Theorem \ref{thm:sublinear} gives the LDP
\[
-\underset{y\in {\rm int}(A)}\inf |y|\le  \underset{n\to\infty}\liminf \tfrac{1}{n}\log\mathbb{P}(X_n\in A)\le
\underset{n\to\infty}\limsup \tfrac{1}{n}\log\mathbb{P}(X_n\in A)\le -\underset{y\in {\rm cl}(A)}\sup |y|,  
\]
where the lower bound is not trivial whenever ${\rm int}(A)\neq \emptyset$.  
Moreover, we obtain the LP
\[
\underline{\psi}(f)=\overline{\psi}(f)=\underset{n\to\infty}\lim\frac{1}{n}\log\int_{-\infty}^\infty \tfrac{n}{2}e^{n (f(x)-|x|)}dx=
\underset{x\in\R}\sup\{f(x)-|x|\},
\]
for all $f\in C_{\overline{\psi}}(E)$. 
The present example illustrates that G\"{a}rtner-Ellis theorem does not capture the lower bound in the LDP when the rate function has large parts that are not exposed by hyperplanes. 
In contrast, by considering other types of exposing functions, for instance, inverted v-shaped functions as in this case, we can produce richer classes of exposing points and prove a LDP by means of  Theorem~\ref{thm:sublinear}.  

Although G\"{a}rtner-Ellis theorem does not capture the lower bound in the LDP, we may try other known methods to prove it. 
Next, we briefly discuss some of these methods.  
We note that, in this simple one-dimensional case, the limit representation \cite[Theorem 4.1.18]{dembo} yields the rate function $I(x)=|x|$, but we would still need to prove the existence of a LDP to apply this result.  
Since $(X_n)_{n\in\mathbb{N}}$ is exponentially tight, we may apply Bryc's theorem \cite[Theorem 4.4.2]{dembo}.  
However, we need to verify that $\overline{\psi}(f)=\underline{\psi}(f)$ for all $f\in C_b(\mathbb{R})$ to derive the lower LDP bound, which is not immediate.   
Alternatively, as proven in~\cite{comman2}, one can replace $C_b(\mathbb{R})$ in Bryc's theorem by an algebra $\mathcal{A}$ of continuous functions separating the points of $\mathbb{R}$, or any well-separating class $\mathcal{A}$  (i.e.~$\mathcal{A}$ contains the constant functions, is closed under finite infima, and separates points of $\mathbb{R}$).  
However, this methodology leads again to rather intricate classes of functions where we need to verify $\overline{\psi}(f)=\underline{\psi}(f)$. 
In contrast, our method leads to a family of inverted v-shaped functions for which is directly checked that $\overline{\psi}(f)=\underline{\psi}(f)=0$.

\end{example}

\begin{remark}   
The main result of \cite{tan} is a version of G\"{a}rner Ellis theorem for sequences of sublinear expectations on a finite dimensional setting, which is proven by adapting the proof of the standard case. 
This result is a particular instance of Theorem \ref{thm:sublinear} for the special choices  $E=\R^d$ and $\mathcal{H}=(\R^d)^\ast=\R^d$, taking into account (2) in Examples~\ref{examples:linear}. 
As illustrated in Example \ref{ex:classical}, this setting does not cover all $\R^d$ cases in which a LDP exists.    
In addition, the authors of~\cite{tan} consider sublinear expectations of the particular form $\mathcal{E}_n(f)=\sup_{Q\in\mathcal{P}}\mathbb{E}_Q[f(X_n)]$ for some set $\mathcal{P}$ of probabilities measures and a sequence of random variables $(X_n)_{n\in\N}$. 
In particular, this implies that the sublinear $\mathcal{E}_n$ is continuous from below due to the monotone convergence theorem. 
Such a continuity condition is not needed in the present approach. 
Also, among other restrictions it is assumed in \cite{tan} that the limit $\lim_{n\to\infty}\frac{1}{n}\log \mathcal{E}_n(exp(n f))$ exists for all $f\in E^\ast$ (see \cite[Assumption 3.1]{tan}), which was not needed here.      
Furthermore, in \cite{tan} it is not derived the Laplace principle~\eqref{eq:LP2}.
\end{remark}

\begin{appendix}

\section{Pairs of monetary risk measures}
As a consequence of the duality bounds provide in \cite{kupper2}  (see Theorem \ref{thm:dualityBounds}) and in line with \cite[Proposition 5.2]{kupper} we have the following.
\begin{proposition}\label{prop:LDPimpliesLP}
Suppose that $\underline{\phi},\overline{\phi}$ is a pair of monetary risk measures such that $\overline{\phi}$ is weakly maxitive and $\underline{\phi}(f)\le \overline{\phi}(f)$ for all $f\in B(E)$. 
If the pair $\underline{\phi},\overline{\phi}$ satisfies the LDP with rate function $I(\cdot)$, then  the pair $\underline{\phi},\overline{\phi}$ satisfies the LP with rate function $I(\cdot)$.
\end{proposition}
 \begin{proof}
 Suppose that the pair $\underline{\phi},\overline{\phi}$ satisfies the LDP with rate function $I(\cdot)$, and fix $f\in C_{\overline{\phi}}(E)$. Since $\overline{\phi}$ is weakly maxitive, due to Theorem \ref{thm:repII} we have that $\phi_{\overline{J}}(f)=\overline{\phi}(f)$.  
Then, by Theorem~\ref{thm:dualityBounds}, we have
\begin{equation}\label{eq:kkLP}
\overline{\phi}(f)=\phi_{\overline{J}}(f)=\underset{x\in E}\sup\{f(x)-I(x)\}.
\end{equation} 
Now, given $x\in E$ and $\delta>0$, since $f$ is upper semicontinuous, there exists $U\in\mathcal{U}_x$ such that $\inf_{y\in U} f(y)\ge f(x)-\delta$. 
Then, by monotonicity and translation invariance
\begin{align*}
\underline{\phi}(f)&\ge \underline{\phi}(f1_U -\infty 1_{U^c})\ge f(x)-\delta + \underline{J}_U\\
&\ge f(x)-\delta -\underset{y\in U}\inf I(y)\ge f(x)-\delta - I(x).
\end{align*}
Letting $\delta \downarrow 0$ and taking the supremum over all $x\in E$ also yields
\[
\underline{\phi}(f)\ge \underset{x\in E}\sup\{f(x)-I(x)\}.
\]
This along with \eqref{eq:kkLP} shows that the pair $\underline{\phi},\overline{\phi}$ satisfies the LP with rate function $I(\cdot)$. 
 \end{proof}

\end{appendix}

\end{document}